\documentclass[10pt]{amsart}
\usepackage{amsmath}
\usepackage{amsthm}
\usepackage{amsfonts}
\usepackage{amssymb}
\usepackage{amscd}

\theoremstyle{plain}
\newtheorem{maintheorem}{Theorem}

\newtheorem{theorem}{Theorem}[section]
\newtheorem{proposition}[theorem]{Proposition}
\newtheorem{lemma}[theorem]{Lemma}

\theoremstyle{definition}
\newtheorem{remark}[theorem]{Remark}

\newcommand{\cL}{{\mathcal L}}

\newcommand{\cP}{{\mathcal P}}

\newcommand{\cZ}{{\mathcal Z}}
\newcommand{\fP}{{\mathfrak P}}

\newcommand{\vep}{\varepsilon}

\newcommand{\RR}{\mathbb{R}}
\newcommand{\ZZ}{\mathbb{Z}}
\newcommand{\PP}{\mathbb{P}}

\newcommand{\NN}{\mathbb{N}}

\newcommand{\GL}{\operatorname{GL}}

\newcommand{\quand}{\quad\text{and}\quad}

\newcommand{\pF}{\PP(F)}
\newcommand{\phF}{\PP(\hat F)}

\newcommand{\pu}{u}
\newcommand{\pv}{v}
\newcommand{\pw}{w}
\newcommand{\pz}{z}
\newcommand{\vu}{\underline{u}}
\newcommand{\vv}{\underline{v}}
\newcommand{\vw}{\underline{w}}
\newcommand{\proj}{\operatorname{proj}}

\begin{document}

\title{Lyapunov exponents of linear cocycles over Markov shifts}
\author{Ela{\'\i}s C. Malheiro and Marcelo Viana}
\address{IMPA, Est. D. Castorina 110 \\ 22460-320 Rio de Janeiro, RJ, Brazil}
\email{elaisc@impa.br, viana@impa.br}
\thanks{The authors are supported by CNPq, FAPERJ and PRONEX-Dynamical Systems.}
\date{\today}

\begin{abstract}
The Lyapunov exponents of $\GL(2)$-cocycles over Markov shifts depend continuously on
the underlying data, that is, on the matrix coefficients and the Markov measure
transition probabilities.
\end{abstract}

\maketitle


\section{Introduction}\label{s.introduction}

The notion of Lyapunov exponent goes back to the late 19th century work of A. M. Lyapunov $[12]$
on the stability of solutions of differential equations: the so-called first method of Lyapunov amounts to
saying that, under suitable conditions, one has exponential stability whenever the exponents are all negative.

It turns out that these numbers encode very important information on the behavior of the dynamical system.
A striking illustration of this fact is the theory of (non-uniformly) hyperbolic systems,
initiated by Oseledets $[14]$ and Pesin $[15, 16]$. In a few words, it asserts that smooth
diffeomorphisms and flows whose Lyapunov exponents do not vanish exhibit a very rich geometric structure,
including invariant stable and unstable laminations that are absolutely continuous
(see for instance $[5, \text{Appendic C}]$), from which a refined description of the dynamics can be drawn.

In this work we are particularly concerned with the dependence of Lyapunov exponents on the underlying
system. In addition to its intrinsic interest, this problem is also connected to the issue of when do
Lyapunov exponents vanish. For example Bourgain, Jitomirskaya $[7, 6]$ proved that the Lyapunov
exponent of certain Schr\"odinger cocycles with analytic potential vary continuously with the energy
parameter and used that fact to obtain an explicit lower bound for the Lyapunov exponent. From this they
deduced that such cocycles have Anderson localization.

While there is a number of other situations where Lyapunov exponents are known to vary continuously, and
even analytically, under perturbations of the system (see $[17, \text{Section 10.6}]$ and the references therein),
it is also known that in general their behavior is rather wild. Again, an especially striking illustration
is the discovery, due to Ma\~n\'e $[13]$, that in the realm of $C^1$ area-preserving surface diffeomorphisms
the only continuity points for the Lyapunov exponents are the Anosov maps (which exist only on the torus)
and those maps whose exponents vanish almost every point. A complete proof was given by Bochi $[1]$
and this has also been extended to arbitrary dimension, by Bochi, Viana $[3, 2]$.

The derivative of a $C^1$ diffeomorphism is a $C^0$ linear cocycle, of course. For cocycles that are
better than just continuous, the picture seems to be richer and is far from being understood.
In this direction, the second author conjectured around 2004 that the Lyapunov exponents vary continuously
restricted to the subset of fiber-bunched H\"older-continuous $\GL(2)$-cocycles.
For definitions and more precise statements, see $[5, \text{Section 12.4}]$ and $[17, \text{Section 10.6}]$.
The fiber-bunching condition can not be omitted, as observed in $[17, \text{Section 9.3}]$. But continuity
should hold within any family of cocycles admitting invariant holonomies.

The first result was obtained in 2009 by Bocker, Viana $[4]$, who proved continuity for locally
constant cocycles over a Bernoulli shift. It is this result that we now extend to the Markov case:
\emph{The Lyapunov exponents of $\GL(2)$-cocycles over Markov shifts depend continuously on the matrix
coefficients and the transition probabilities.} The complete statement is given in
Theorem~\ref{t.main} below.

Our method is very different from the one in $[4]$ and may be viewed as an extension of the
method used in $[17, \text{Chapter 10}]$ to give an alternative proof of the theorem of Bocker-Viana.
In either case, the starting point is the notion of stationary measure and the fact that the stationary
measures completely determine the Lyapunov exponents (see Furstenberg $[8]$ and Ledrappier $[11]$).
Then, both arguments are based on estimating the stationary measures of nearby cocycles, to prove that
they can not accumulate too much on the neighborhood of a single point. This, they achieve in very
different ways.

Now, the notion of stationary measure does not really make sense outside of the Bernoulli case.
In the present Markov setting, we are able to substitute it with the notion of \emph{stationary measure vector,}
which we introduce here. While this is a fairly straightforward extension of the notion of stationary measure,
stationary measure vectors are a lot more difficult to handle because, in a few words, we need the different
vector components to be quite homogeneous. Indeed, a good part of our efforts in Sections~\ref{s.proof2}
through~\ref{s.proof3} is devoted to establishing such homogeneity.

This notion of stationary measure vector has a natural counterpart in the general setting of the conjecture
mentioned previously, and the methods we develop here should be useful for proving that conjecture.
But the analysis of stationary measure vectors in such generality has yet to be carried out in full.

This work is organized as follows. In Section~\ref{s.statementofresults} we give the precise context and
statement of our main result. In Section~\ref{s.measurevectors} we define stationary measure vector and
prove a few useful properties. In Sections~\ref{s.invariantsubspaces} and~\ref{s.probabilisticrepellers}
we adapt to our setting an argument of Furstenberg, Kifer $[10]$ showing that any possible discontinuity
point must exhibit an invariant subspace that is a kind of repeller for the action of the cocycle in projective
space. Our task is then to show that continuity holds even in the presence of such a repeller.
In Section~\ref{s.proofmain} we reduce this task to proving a key statement, Theorem~\ref{t.key}.
In Section~\ref{s.proof1} we recall the notion of energy of a measure, which has a central role in our
arguments. Then, in Sections~\ref{s.proof2} and~\ref{s.proof3}, we use it to prove Theorem~\ref{t.key},
thus completing the proof.

\section{Statement of main result}\label{s.statementofresults}

Let $X=\{1, \dots, q\}$ and $f:M\to M$ be the shift map on $M=X^\NN$.
We use $x=(i_n)_n$ to represent a generic element of $M$.
For any $m\ge 0$, $n\ge 1$ and $j_0, \dots, j_{n-1} \in X$, denote
$$
[m;j_0, \dots, j_{n-1}] = \{x\in M: i_{m+s} = j_s \text{ for $0 \le s \le n-1$}\}.
$$

Let $P=(P_{i,j})_{i,j\in X}$ be a stochastic matrix and assume that $P$ is \emph{aperiodic},
that is, there exists $N\ge 1$ such that $P^N_{i,j}>0$ for all $i, j \in X$.
By Perron-Fr\"obenius, there exists a unique vector $p=(p_i)_{i\in X}$ such that
$$
p_i>0 \quand \sum_{s\in X} p_s =1 \quand \sum_{s \in X} p_s P_{s,j} = p_j
\quad\text{for every $i, j \in X$.}
$$
Let $\mu$ be the associated Markov measure on $M$: given any $m\ge 0$ and $n\ge 1$,
$$
\mu([m;j_0, \dots, j_{n-1}]) = p_{j_0} P_{j_0,j_1} \cdots P_{j_{n-2},j_{n-1}}
\quad\text{for every $j_0, \dots, j_{n-1}\in X$.}
$$
Then $\mu$ is invariant under $f$ and the system $(f,\mu)$ is mixing.

We use $\pu, \pv, \pw$ to represent the lines spanned by vectors $\vu, \vv, \vw \in\RR^d\setminus\{0\}$.
Let $d\ge 2$ and $A:X\to\GL(d)$. Consider the linear cocycle
$$
F:M\times\RR^d\to M\times\RR^d,
\quad (x, \vv) \mapsto (f(x), A(i_0)\vv)
$$
and the associated projective cocycle
$$
\pF:M\times\PP(\RR^d)\to M\times\PP(\RR^d),
\quad (x,\pv) \mapsto (f(x), A(i_0)v).
$$
Here and in what follows we use the same notation for an element of $\GL(d)$ and
for its action in projective space.

Denote $A^n(x)=A(i_{n-1}) \cdots A(i_1)A(i_0)$, for each $x\in M$ and $n\ge 1$.
By Furstenberg, Kesten $[9]$, there exist numbers $\lambda_+(A,P) \ge \lambda_-(A,P)$ such that
$$
\begin{aligned}
\lambda_+(A,P) & = \lim_n \frac 1n \log\|A^n(x)\|\\
\lambda_-(A,P) & = \lim_n \frac 1n \log\|A^n(x)^{-1}\|^{-1}
\end{aligned}
\quad\text{for $\mu$-almost every $x\in M$.}
$$
They are called the \emph{extremal Lyapunov exponents} of $F$ relative to $\mu$.
We are going to prove:

\begin{maintheorem}\label{t.main}
For $d=2$, the functions $(A,P) \mapsto \lambda_\pm(A,P)$ are continuous,
relative to the natural (coefficient-induced) topology in the space of pairs $(A,P)$.
\end{maintheorem}

The special case of this theorem when $\mu$ is a Bernoulli measure, that is, when
$P_{i,j} = p_j$ for every $i, j \in X$, is contained in the main result of Bocker, Viana $[4]$.
Recently, Avila, Eskin, Viana announced that, still in the Bernoulli case,
the theorem holds in arbitrary dimension $d\ge 2$.
Although their paper has not yet appeared, the $2$-dimensional version of their approach
has been presented in $[17, \text{Chapter 10}]$ and that was the starting point for our analysis
of the Markov case.

It is well known (see $[17, \text{Section 9.1}]$) that the functions $\lambda_+$ and $\lambda_-$ are,
respectively, upper semi-continuous and lower semi-continuous.
Thus, continuity holds automatically at any point such that $\lambda_-(A,P) = \lambda_+(A,P)$.
In what follows we assume that $\lambda_-(A,P) < \lambda_+(A,P)$.
We are going to prove continuity for $\lambda_+$: to deduce the corresponding statement for
$\lambda_-$ just use Lemma~\ref{l.somaOK} below.

When the stochastic matrix $P$ is just \emph{irreducible} (for every $i, j \in X$ there exists $N\ge 1$
such that $P^N_{i,j}>0$), our arguments remain valid to show that $A \mapsto \lambda_\pm(A,P)$ are continuous.
Moreover, one recovers continuous dependence on both variables, as in the conclusion of the theorem,
under the additional assumption that the probability vector $p$ is chosen depending continuously on
the matrix in a neighborhood of $P$ (this is automatic in the aperiodic case).


Most steps in the proof of Theorem~\ref{t.main} are valid for arbitrary dimension $d$.
We specialize to $d=2$ only when necessary.

\section{Measure vectors}\label{s.measurevectors}

The following simple notions have a central part in our arguments. We consider vectors $\eta=(\eta_i)_{i\in X}$
where each $\eta_i$ is a (positive) measure on $\PP(\RR^d)$.
With any such vector, we associate the \emph{skew-product measure} $m=\mu\ltimes\eta$ defined on $M\times\PP(\RR^d)$ by
\begin{equation}\label{eq.skew-product}
m = \sum_{i\in X} \big(\mu \mid [0;i]\big) \times \eta_i.
\end{equation}
We call $\eta$ a \emph{unit vector} if every $\eta_i$ is a probability measure on $\PP(\RR^d)$.
Then $m$ is a probability measure on $M\times\PP(\RR^d)$.

Let $\cP$ be the operator defined in the space of measure vectors $\eta$ by
\begin{equation}\label{eq.operator1}
(\cP\eta)_j(D) = \sum_{i\in X} \frac{p_iP_{i,j}}{p_j}\eta_i\big(A(i)^{-1}(D)\big)
\end{equation}
for any $j\in X$ and any measurable $D\subset\PP(\RR^d)$. We say that $\eta$ is \emph{$P$-stationary}
if it is a fixed point for $\cP$, that is, if it satisfies
$$
\sum_{i\in X} p_i P_{i,j} \eta_i(A(i)^{-1}(D)) = p_j\eta_j(D)
$$
for every $j\in X$ and every measurable $D\subset\PP(\RR^d)$.
In the Bernoulli case, this means that $\sum_{i\in X} p_i \eta_i(A(i)^{-1}(D)) = \eta_j(D)$ for every $j$
and every $D$. Then the $\eta_j$ are all equal, and one recovers the more usual notion of stationary measure.

\begin{proposition}\label{p.1}
A unit vector $\eta$ is $P$-stationary if and only if the probability measure $m=\mu\ltimes\eta$ is $\pF$-invariant.
\end{proposition}

\begin{proof}
For any $j, j_2, \dots, j_n \in X$ and any measurable set $D\subset \PP(\RR^d)$,
$$
\pF^{-1}\big([0;j, j_2, \dots, j_n] \times D\big)
= \bigcup_{i\in X} [0;i,j, j_2, \dots, j_n] \times A(i)^{-1}(D).
$$
Thus, using the definition~\eqref{eq.skew-product},
$$
\begin{aligned}
\pF_*m\big([0;j, j_2, \dots, j_n] \times D \big)
& = \sum_{i\in X} \mu\big([0;i,j, j_2, \dots, j_n]\big)\eta_i\big(A(i)^{-1}(D)\big).
\end{aligned}
$$
The right-hand side may be written as
$$
\sum_{i\in X} \frac{p_i P_{i,j}}{p_j} \mu([0;j, j_2, \dots, j_n])\eta_i\big(A(i)^{-1}(D)\big)
= \mu([0;j, j_2, \dots, j_n])(\cP\eta)_j(D).
$$
This proves that $\pF_* \big(\mu \ltimes \eta\big) = \mu \ltimes \cP\eta$.
The claim is a consequence.
\end{proof}

The set of unit vectors $\eta=(\eta_i)_{i\in X}$ has a natural topology induced by the weak$^*$ topology
in the space of probability measures on $\PP(\RR^d)$.

\begin{proposition}\label{p.2}
The set of $P$-stationary unit vectors $\eta$ is non-empty, compact and convex.
\end{proposition}

\begin{proof}
First, note that the operator $\cP$ is continuous. Indeed, if $(\eta_n)_n$ is a sequence of unit vectors
converging to some $\eta$ then, given any continuous function $\phi:\PP(\RR^d)\to\RR$,
$$
\int \phi \, d(\cP\eta_n)_j = \sum_{i\in X} \frac{p_iP_{i,j}}{p_j} \int (\phi \circ A(i)) \, d\eta_{n, i}
$$
for every $j \in X$ and every $n\ge 1$. Since $\eta_{n,i}\to\eta_i$ in the weak$^*$ topology,
the expression on the right-hand side converges to
$$
\sum_{i\in X} \frac{p_iP_{i,j}}{p_j} \int (\phi \circ A(i)) \, d\eta_i
= \int \phi \, d(\cP\eta)_j.
$$
This proves that $(\cP\eta_n)_j$ converges to $(\cP\eta)_j$ in the weak$^*$ topology, for every $j\in X$,
as claimed.

Next, let $\xi$ be an arbitrary unit vector and define, for $j\in X$ and $n\ge 1$,
$$
\eta_{n, j}=\frac{1}{n}\sum_{l=0}^{n-1}(\cP^{l}\xi)_j.
$$
Let $\eta$ be an accumulation vector of the sequence $(\eta_n)_n$, where $\eta_n=(\eta_{n, j})_{j\in X}$.
Observe that
$$
(\cP\eta_n)_j
=\frac{1}{n}\sum_{l=0}^{n-1}(\cP^{l+1}\xi)_j
=\eta_{n,j} + \frac{1}{n} (\cP^n\xi)_j - \frac 1n \xi_j.
$$
So, given any continuous function $\phi:M\to\RR$,
$$
\lim_n \int \phi \, d(\cP\eta_n)_j
=\lim_n \big(\int \phi \, d\eta_{n,j}+\frac{1}{n} \int \phi \, d(\cP^{n}\xi)_j- \frac{1}{n} \int \phi \, d \xi_j\big)
=\int \phi \, d\eta_j.
$$
This means that $\eta_j = \lim_n (\cP\eta_n)_j$. We also have $\lim_n(\cP\eta_n)_j=(\cP\eta)_j$, by the continuity of
the operator $\cP$. It follows that $(\cP\eta)_j = \eta_j$ for every $j\in X$, which means that $\eta$ is $P$-stationary.
This proves that stationary vectors do exist.

Let $(\eta_n)_n$ be a sequence of $P$-stationary unit vectors converging to $\eta$.
For each $j\in X$ and $n \ge 1$, we have $(\cP\eta_n)_j=\eta_{n, j}$ and $\lim_n \eta_{n, j}=\eta_j$.
As we have already seen, $\lim_n(\cP\eta_n)_j=(\cP\eta)_j$.
Then $(\cP\eta)_j=\eta_j$, which means that $\eta$ is a $P$-stationary unit vector.
This implies the set of $P$-stationary unit vectors is closed and, consequently, is compact.

Let $\eta_1$ and $\eta_2$ be $P$-stationary unit vectors and consider $\eta=a_1\eta_1 + a_2\eta_2$.
For each $j\in X$, we have
$$
\eta_j(D)=a_1\eta_{1, j}(D) + a_2\eta_{2, j}(D)=a_1(\cP\eta_1)_j(D)+a_2(\cP\eta_2)_j(D)\\
$$
$$
=\sum_{i\in X} a_1\frac{p_iP_{i,j}}{p_j}\eta_{1, i}\big(A(i)^{-1}(D)\big)
+\sum_{i\in X} a_2\frac{p_iP_{i,j}}{p_j}\eta_{2, i}\big(A(i)^{-1}(D)\big)\\
$$
$$
=\sum_{i\in X} \frac{p_iP_{i,j}}{p_j}\eta_ i\big(A(i)^{-1}(D)\big)=(\cP\eta)_j(D).
$$
This proves that the set of $P$-stationary unit vectors is convex.
\end{proof}

\begin{proposition}\label{p.3}
The set $\{(P,\eta): \eta \text{ is a $P$-stationary unit vector}\}$ is closed.
\end{proposition}

\begin{proof}
Let $(P_k, \eta_k)_{k>0}$ be a sequence converging to $(P, \eta)$, where each $\eta_k$ is a $P_k$-stationary unit vector.
This means that, for any $j\in X$ and $k \ge 1$
$$
\eta_{k, j}
= (\cP_k\eta_k)_j
= \sum_{i\in X} \frac{p_{k, i}P_{k,i,j}}{p_{k, j}} A_k(i)_*\eta_{k, i}
$$
Then, given any continuous function $\phi:\PP(\RR^d)\to\RR$,
$$
\int \phi \, d\eta_{k,j}
= \sum_{i\in X} \frac{p_{k,i}P_{k,i,j}}{p_{k,j}} \int (\phi \circ A_k(i)) \, d\eta_{k,i}.
$$
The left-hand side converges to $\int \phi \, d\eta_j$ and the right-hand side converges to
$$
\sum_{i\in X} \frac{p_{i}P_{i,j}}{p_j} \int (\phi \circ A(i)) \, d\eta_{i}
= \int \phi \, d(\cP\eta)_j,
$$
because $p_{k, i}$ and $P_{k,i,j}$ and $\eta_{k ,j}$ and $A_k(i)$ converge, respectively,
to $p_i$ and $P_{i,j}$ and $\eta_j$ and $A(i)$.
Thus, $\eta_j = (\cP\eta)_j$ for every $j\in X$. This means that $\eta$ is a $P$-stationary unit vector.
\end{proof}

A $P$-stationary vector $\eta$ is \emph{atomic} if there exists $\pv\in\PP(\RR^d)$ such that $\{\pv\}$ has
positive $\eta_i$-measure for some $i\in X$. Then $\pv$ is an \emph{atom} of $\eta$.

\begin{proposition}\label{p.5}
If $\eta$ is an atomic $P$-stationary unit vector then there exists a finite non-empty set
$\cL\subset\PP(\RR^d)$ such that $A(i)\pv\in\cL$ and $\eta_i(\{\pv\})>0$ for every $\pv\in\cL$ and
every $i\in X$.
\end{proposition}

\begin{proof}
Let $\delta$ be the largest $\eta_i$-measure of any point $\pv \in \PP(\RR^d)$ for any $i \in X$.
By assumption, $\delta>0$.
Let $\cL$ be the set of all $\pw\in\PP(\RR^d)$ such that $\eta_j(\{\pw\})=\delta$ for some $j \in X$.
Notice that $\cL$ is  non-empty and $\#\cL \le \# X/\delta < \infty$, because the $\eta_i$ are
probability measures. Since
$$
\eta_j\big(\{\pw\}\big) = \sum_{i\in X} \frac{p_i P_{i,j}}{p_j} \eta_i\big(\{A(i)^{-1}\pw\}\big)
\quand
\sum_{i\in X} \frac{p_i P_{i,j}}{p_j} = 1,
$$
we have that
\begin{equation}\label{eq.average}
\eta_j\big(\{\pw\}\big)=\delta \quad \Rightarrow \quad \eta_i\big(\{A(i)^{-1}\pw\}\big)=\delta
\quad\text{for every $i\in X$.}
\end{equation}
In particular, $A(i)^{-1}(\cL) \subset \cL$ for every $i\in X$.
Since $\cL$ is finite, this means that $A(i)^{-1}(\cL)=\cL$ for every $i\in X$.
So, given any $\pv\in\cL$ and $i\in X$, we can find $\pw\in \cL$ such that $\pv=A(i)^{-1}\pw$.
By \eqref{eq.average}, it follows that $\eta_i(\{\pv\})=\delta$.
\end{proof}

Consider the continuous function $\Phi:M\times\PP(\RR^d)\to\RR$ defined by
\begin{equation}\label{eq.Phi}
\Phi(x,\pv)
= \log \frac{\|A(x)\vv\|}{\|\vv\|}
= \log \frac{\|A(i_0)\vv\|}{\|\vv\|}.
\end{equation}
Clearly,
$$
\int \Phi \, d(\mu\ltimes\eta) = \sum_{i\in X} p_i \int \log \frac{\|A(i)\vv\|}{\|\vv\|} \, d\eta_i(v).
$$
Results closely related to the next proposition are discussed in $[17, \text{Chapter 6}]$,
mostly in the Bernoulli context.
The proposition is true in any dimension, but we choose to restrict ourselves to the case $d=2$,
which is simpler and suffices for all our purposes.

\begin{proposition}\label{p.4}
$$
\lambda_+(A,P) = \max\big\{\int \Phi \, d(\mu\ltimes\eta): \eta \text{ is a $P$-stationary unit vector}\big\}.
$$
\end{proposition}

\begin{proof}
Let $m$ be an arbitrary $\pF$-invariant probability measure on $M\times\PP(\RR^2)$.
For every $(x,\pv)$ and every $n\ge 1$,
$$
\frac 1n \sum^{n-1}_{j=0} \Phi(\pF^{j}(x,\pv))
= \frac 1n \log \frac{\|A^{n}(x)\vv\|}{\|\vv\|}
\le \frac 1n\log\|A^{n}(x)\|.
$$
Let $\tilde\Phi(x,\pv)$ denote the limit of the left-hand side. This is well-defined $m$-almost everywhere,
by the ergodic theorem, and the previous inequality implies that $\tilde\Phi(x,\pv)\le\lambda_+(A,P)$. So,
\begin{equation}\label{eq.ineq1}
\int \Phi \, dm
= \int \tilde\Phi \, dm
\le \lambda_+(A,P).
\end{equation}
In particular, $\lambda_+(A,P)$ is an upper bound for the integral of $\Phi$ relative
to any $\pF$-invariant measure of the form $m=\mu\ltimes\eta$.
We are left to prove that the equality in \eqref{eq.ineq1} does hold for some such measure.

It is convenient to consider the natural extension of the cocycle $F$
$$
\hat F:\hat M \times \RR^2 \to \hat M \times \RR^2,
\quad \hat F(\hat x,\vv) = (\hat f(\hat x), A(i_0)\vv),
$$
where $\hat f:\hat M \to \hat M$ is the shift on $\hat M = X^\ZZ$.
Denote by $\hat \mu$ the $\hat f$-invariant Markov measure defined on $\hat M$ by the stochastic matrix $P$.
Clearly, $(\hat F, \hat \mu)$ has the same Lyapunov exponents as $(F,\mu)$.
Since we assume that $\lambda_-(A,P) < \lambda_+(A,P)$, the Oseledets decomposition of $(\hat F, \hat \mu)$
has the form
$$
\RR^2 = E^u_{\hat x} \oplus E^s_{\hat x}, \quad\text{with } \dim E^u_{\hat x} = \dim E^s_{\hat x} = 1
$$
(the invariant bundles $E^u$ and $E^s$ are associated with the Lyapunov exponents $\lambda_+(A,P)$
and $\lambda_-(A,P)$, respectively) for $\hat\mu$-almost every $\hat x=(i_n)_{n\in\ZZ}\in\hat M$.



Let $\phF:\hat M\times\RR^2\to \hat M\times\RR^2$ be the projective cocycle associated with $\hat F$.
Let $\pi:\hat M \to M$ and $\Pi:\hat M \times \PP(\RR^2) \to M \times \PP(\RR^2)$, $\Pi(\hat x, \pv)=(\pi(\hat x),\pv)$
and  $\pi_1: M \times \PP(\RR^2) \to M$ and $\hat\pi_1: \hat M \times \PP(\RR^2) \to \hat M$ be the canonical projections.
If $\hat m$ is a $\phF$-invariant measure projecting down to $\hat \mu$ (under $\hat\pi_1$) then $m=\Pi_*\hat m$
is an $\pF$-invariant measure projecting down to $\mu$ (under $\pi_1$): that is an immediate consequence of the relations
$\pF \circ \Pi =\Pi \circ \phF$ and $\pi_1 \circ \Pi = \pi \circ \hat\pi_1$.

Take $\hat m$ to be the probability measure on $\hat M \times \PP(\RR^2)$ defined by
$$
\hat m (C \times D) = \hat\mu\big(\{\hat x \in C: E^u_{\hat x} \in D\}\big),
$$
for any measurable sets $C \subset \hat M$ and $D \subset \PP(\RR^2)$.
In other words, $\hat m$ projects down to $\hat\mu$ and its conditional probabilities along the fibers are given by the
Dirac masses at the unstable subspace $E_{\hat x}^u$. It is clear that $\hat m$ is $\phF$-invariant, since the measure
$\hat\mu$ is $\hat f$-invariant and the unstable bundle $E^u$ is $\hat F$-invariant.
So, $m$ is a $\pF$-invariant probability measure.

Next, we prove that $m$ is a skew-product measure. The key observation is that, since the unstable subspace
$E^u_{\hat x}$ is determined by the negative iterates of the cocycle alone, each the conditional probability
$$
\text{$\hat m_{\hat x} = \delta_{E^u_{\hat x}}$ depends only on the negative component $x^-=(i_n)_{n<0}$ of $\hat x\in\hat M$.}
$$
Moreover, the Markov measure $\hat \mu$ is a product measure restricted to each cylinder $[0;i]\subset\hat M$:
there are probability measures $\mu_i^-$ and $\mu_i^+$ on $M^-=X^{\{n<0\}}$ and $M^+=X^{\{n>0\}}$, respectively,
such that
$$
\hat \mu \mid [0;i] = p_i \, (\mu_i^- \times \mu_i^+).
$$
Given $i\in X$ and measurable sets $C^-\subset M^-$, $C^+ \subset M^+$ and $D\subset\PP(\RR^2)$,
$$
\begin{aligned}
\hat m(C^-\times \{i\} \times C^+ \times D)
& = \int_{C^-\times \{i\} \times C^+} \hat m_{\hat x}(D) \, d\hat\mu(\hat x)\\
& = p_i \mu_i^+(C^+) \int_{C^-} \hat m_{x^-}(D) \, d\mu_i^-(x^-)
\end{aligned}
$$
whereas
$$
\hat\mu(C^- \times\{i\} \times  C^+) = p_i \, \mu_i^-(C^-) \mu_i^+(C^+).
$$
In particular, the quotient between these two numbers is independent of $C^+$. Any subset $C$ of the one-sided cylinder
$[0;i]\subset M$ may be written as $C=\{i\} \times C^+$ with $C^+\subset M^+$. Then,
$$
\eta_i(D)
= \frac{m(C \times D)}{\mu(C)}
= \frac{\hat m(M^- \times\{i\} \times C^+ \times D)}{\hat\mu(M^- \times\{i\} \times C^+)}
$$
is independent of $C$ and defines a probability measure $\eta_i$ on $\PP(\RR^2)$.
Let $\eta$ be the measure vector formed by these probabilities.
For any measurable $C\subset M$ and $D\subset\PP(\RR^2)$,
$$
m\big(C \times D\big)
= \sum_{i\in X} m((C \cap [0;i]) \times D)
= \sum_{i\in X} \mu(C \cap [0;i]) \times \eta_i(D).
$$
This proves that $m = \mu \ltimes \eta$.

Finally, consider the function $\Psi: \hat M \times \PP(\RR^{d}) \to \RR$ defined by
$$
\Psi(\hat x,\pv)
= \log\frac{||A(i_{0})\vv||}{||\vv||}
= \Phi(\pi(\hat x), \pv).
$$
Clearly, $\int \Phi \, dm = \int \Psi \, d\hat m = \int \tilde \Psi \, d\hat m$,
where $\tilde\Psi$ is the Birkhoff time average:
$$
\tilde\Psi(\hat x, \pv)
= \lim_n \frac 1n \sum_{j=0}^{n-1}\Psi(\phF^j(\hat x,\pv))
= \lim_n \frac 1n \log\frac{\|A^n(x)\vv\|}{\|\vv\|}.
$$
Since the conditional probabilities of $\hat m$ are concentrated on the Oseledets
subspaces $E^u_{\hat x}$, we have $\tilde\Psi(\hat x, \pv)=\lambda_+(A,P)$ for
$\hat m$-almost every $(\hat x,\pv)$. It follows that $\int \Phi \, dm = \lambda_+(A,P)$.

This completes the proof of Proposition~\ref{p.4}.
\end{proof}

\section{Invariant subspaces}\label{s.invariantsubspaces}

Consider any sequence $(A_k, P_k)_k$ converging to $(A,P)$. Let $p_k$ be the probability vector associated with each $P_k$.
Moreover,

\begin{lemma}\label{l.somaOK}
$\lambda_+(A_k,P_k) + \lambda_-(A_k,P_k)$ converges to $\lambda_+(A,P) + \lambda_-(A,P)$.
\end{lemma}

\begin{proof}
By the theorem of Oseledets $[14]$ and the ergodicity of $(f,\mu)$,
$$
\lambda_+(A,P)+\lambda_-(A,P)
= \lim_n \frac 1n \log |\det A^n(x)|
= \int \log |\det A| \, d\mu
$$
for $\mu$-almost every $x$, and analogously for $\lambda_+(A_k,P_k) + \lambda_-(A_k,P_k)$
for every $k$. Now observe that
$$
\int \log |\det A_k| \, d\mu_k
= \sum_{i\in X} p_{k,i} \log|\det A_k(i)|
$$
converges to
$$
\sum_{i\in X} p_i \log|\det A(i)|
= \int \log |\det A| \, d\mu
$$
when $k\to\infty$.
\end{proof}

For each $k$, let $\Phi_k:M\times\PP(\RR^d)\to\RR$ be defined as in \eqref{eq.Phi}, with $A_k$ in
the place of $A$. Moreover, let $\eta_k=(\eta_{k,i})_{i\in X}$ be a $P_k$-stationary unit vector
that realizes the largest Lyapunov exponent for $(A_k,P_k)$:
$$
\lambda_+(A_k,P_k)
= \int \Phi_k \, d(\mu_k\ltimes \eta_k)
= \sum_{i\in X} p_{k,i} \int \log\frac{\|A_k(i)\vv\|}{\|\vv\|} \, d\eta_{k,i}(\pv).
$$
Up to restricting to a subsequence, we may suppose that $(\eta_k)_k$ converges to some $\eta$ in the
weak$^*$ topology. By Proposition~\ref{p.3}, the vector $\eta$ is $P$-stationary. Moreover,
$$
\int \Phi_k \, d(\mu_k \ltimes \eta_k) \quad\text{converges to}\quad \int \Phi \, d(\mu\ltimes\eta).
$$
If $\eta$ realizes $\lambda_+(A,P)$ then we are done. In all that follows we suppose that
\begin{equation}\label{eq.contradiction}
\int \Phi \, d(\mu\ltimes\eta) < \lambda_+(A,P).
\end{equation}

\begin{proposition}\label{p.Lexists}
Under assumption \eqref{eq.contradiction}, there exists $L\in\PP(\RR^2)$ such that
\begin{equation}\label{eq.Lexists1}
\eta_i(\{L\}) > 0 \quand A(i)L=L \quad\text{for every $i\in X$}
\end{equation}
and, for $\mu$-almost all $x\in M$,
\begin{equation}\label{eq.Lexists2}
\lim_n \frac 1n \log \|A^n(x)\vv\| =
\left\{\begin{array}{ll}\lambda_-(A,P) & \text{if } \vv \in L\setminus\{0\}\\
                        \lambda_+(A,P) & \text{if } \vv \in \RR^2\setminus L.
       \end{array}\right.
\end{equation}
\end{proposition}

\begin{proof}
By Proposition~\ref{p.1}, the skew-product $m=\mu\ltimes\eta$ is $\pF$-invariant.
So, we may use the ergodic theorem to conclude that
$$
\tilde\Phi\big(x,\pv\big)
= \lim_n \frac 1n \sum_{j=0}^{n-1} \Phi\big(\pF^j(x,\pv)\big)
= \lim_n \frac 1n \log \|A^n(x)\vv\|
$$
exists for $m$-almost every point, is constant on the orbits of $\pF$ and satisfies
$\int \tilde\Phi \, dm = \int \Phi \, dm$. Thus, \eqref{eq.contradiction}
implies that $\tilde\Phi<\lambda_+(A,P)$ on some subset with positive measure for $m$.

\begin{lemma}\label{l.Lexists1}
For $\mu$-almost every $x\in M$ there exists a unique $L_x\in\PP(\RR^2)$ such that
$\tilde\Phi(x,L_x)<\lambda_+(A,P)$. Moreover,
\begin{equation}\label{eq.Lualpha}
L_{f(x)}=A(x)L_x \quad\text{for $\mu$-almost every $x\in M$.}
\end{equation}
\end{lemma}

\begin{proof}
Consider $\cZ=\{(x,\pv)\in M \times \PP(\RR^2): \tilde\Phi(x,\pv) < \lambda_+(A,P)\}$.
This is a measurable, $\pF$-invariant set with positive measure for $m$.
Hence, the projection $Z=\{x\in M:(x,\pv)\in\cZ \text{ for some } \pv\}$ is a measurable
$f$-invariant set with positive measure for $\mu$.
Since $(f,\mu)$ is ergodic, it follows that $\mu(Z)=1$. This proves existence.
Next, given any $x\in M$, suppose that there exist two distinct points $\pv_1$ and $\pv_2$ in
projective space such that $\tilde\Phi(x,\pv_i)<\lambda_+(A,P)$ for $i=1, 2$.
Since every vector in $\RR^2$ may be written as linear combination of $\vv_1$ and $\vv_2$,
it follows that
$$
\lim_n \frac 1n \log \|A^n(x)\| < \lambda_+(A,P).
$$
By the definition of $\lambda_+(A,P)$, this can only happen on a subset with zero measure for $\mu$.
That proves uniqueness.
Property \eqref{eq.Lualpha} follows, since the function $\tilde\Phi$ is constant on the orbits of $\pF$.
\end{proof}

\begin{lemma}\label{l.Lexists2}
There is $L\in\PP(\RR^2)$ such that $L_x=L$ for $\mu$-almost every $x\in M$.
\end{lemma}

\begin{proof}
Let $Z_*$ be the union over $i\in X$ of the set of all $x\in [0;i] \cap Z$ such that $L_x$ is an atom of $\eta_i$.
Observe that $\mu(Z_*)>0$, because
$$
\sum_{i\in X} \int_{[0;i]} \eta_i\big(\{L_x\}\big) \, d\mu(x)
= m\big(\{(x,L_x):x\in M\}\big) = m\big(\cZ\big) >0.
$$
Given any $x\in Z_*$, consider $i, j \in X$ such that $x\in [0;i,j]$. If $P_{i,j}=0$ then $\mu([0;i,j])=0$.
Otherwise, using \eqref{eq.Lualpha},
$$
p_j \eta_j\big(\{L_{f(x)}\}\big)
= \sum_{s \in X} p_s P_{s,j} \eta_s\big(\{A(s)^{-1}L_{f(x)}\}\big)
\ge 
p_i P_{i,j} \eta_i\big(\{L_x\}\big)>0,
$$
and so $f(x)\in Z_*$. This proves that $Z_*$ is $f$-invariant up to measure zero.
By the ergodicity of $(f,\mu)$, it follows that $\mu(Z_*)=1$.

Since the $\eta_i$ are finite measures, they have at most countably many atoms. In particular, there exists $a\in\PP(\RR^d)$
such that $Z_a=\{x\in Z_*: L_x=a\}$ has positive $\mu$-measure. Then, for any $k\ge 1$, we may find $n_k$ and $j_{k,0}, \dots, j_{k,n_k}$
such that
$$
\frac{\mu\big(Z_a \cap [0;j_{k,0}, \dots, j_{k,n_k}]\big)}{\mu\big([0;j_{k,0}, \dots, j_{k,n_k}]\big)}
\ge 1-\frac 1k.
$$
Let $b_k=A(j_{k,n_k-1}) \cdots A(j_{k,0})a$. Then, using \eqref{eq.Lualpha} once more,
$$
\begin{aligned}
\frac{\mu\big(Z_{b_k} \cap [0;j_{k,n_k}]\big)}{\mu\big([0;j_{k,n_k}]\big)}
& \ge \frac{\mu\big(f^{n_k}(Z_a \cap [0;j_{k,0}, \dots, j_{k,n_k}])\big)}{\mu\big(f^{n_k}([0;j_{k,0}, \dots, j_{k,n_k}])\big)}\\
& = \frac{\mu\big(Z_a \cap [0;j_{k,0}, \dots, j_{k,n_k}]\big)}{\mu\big([0;j_{k,0}, \dots, j_{k,n_k}]\big)}
\ge 1-\frac 1k.
\end{aligned}
$$
In particular, for $k\ge 2$,
$$
\mu(Z_{b_k})\ge (1-k^{-1}) p_{j_{k,n_k}}\ge \min_{i\in X} \frac{p_i}{2}>0.
$$
Clearly, there are only finitely many values of $b_k$ for which this inequality can hold. Then, since $j_{k,n_k}$ also takes values
in a finite set $X$, we may choose a subsequence $k_l\to\infty$ for which the values  of $j_{k_l,n_{k_l}}$ and $b_{k_l}$ remain constant.
In this way we find $j\in X$ and $b\in\PP(\RR^d)$ such that
$$
\frac{\mu\big(Z_b \cap [0;j]\big)}{\mu\big([0;j]\big)} = 1.
$$
Equivalently, $L_x=b$ for $\mu$-almost every $x\in[0;j]$. Then, iterating once more, $L_y=A(j)b$ for $\mu$-almost every $y\in X$
(the image of any full $\mu$-measure subset of $[0;j]$ has full $\mu$-measure in $X$).
This implies that $L_x$ takes a unique value on a full $\mu$-measure set, as claimed.
\end{proof}

We are ready to complete the proof of the proposition.
Claim \eqref{eq.Lexists1} follows from $\mu(Z_{A(j)b})=1$ and \eqref{eq.Lualpha}.
By construction, $\tilde\Phi(x,\pv) = \lambda_+(A,P)$ for $\pv \neq L$ and $\tilde\Phi(x,L)< \lambda_+(A,P)$
for $\mu$-almost all $x\in M$. By the Oseledets theorem, $\tilde\Phi$ only takes the values $\lambda_\pm(A,P)$.
It follows that $\tilde\Phi(x,L) = \lambda_-(A,P)$ for $\mu$-almost all $x\in M$.
This gives claim \eqref{eq.Lexists2}.
\end{proof}

\section{Probabilistic repellers}\label{s.probabilisticrepellers}

For any $\alpha\in\GL(d)$, we denote by $D\alpha(\pv)$ the derivative at a point $\pv\in\PP(\RR^d)$ of the action
of $\alpha$ in the projective space. The tangent space to the projective space at $\pv$ is canonically identified
with the hyperplane $v^\perp\subset\RR^d$ orthogonal to $v$. Then the derivative is given by
\begin{equation}\label{eq.proj_der}
D\alpha(\pv)\dot\pv = \frac{\proj_{\alpha(\pv)} \alpha(\dot\pv)}{\|\alpha(\vv)\|/\|\vv\|}
\quad\text{for every } \dot\pv \in \pv^\perp
\end{equation}
where $\proj_\pw: \vu  \mapsto \vu - \vw (\vu\cdot\vw)/(\vw\cdot\vw)$ denotes the orthogonal projection
to the hyperplane orthogonal to $w$. This implies the following estimate:
\begin{equation}\label{eq.derivative_estimate}
\frac{1}{\|\alpha\| \, \|\alpha^{-1}\|}
\le \frac{\|D\alpha(\pv)\dot v\|}{\|\dot v\|}
\le \|\alpha\| \, \|\alpha^{-1}\|
\quad\text{for every non-zero } \dot v \in v^\perp.
\end{equation}
For $n\ge 1$, we will write $DA^n(x,\pv)$ to mean $D\alpha(\pv)$ with $\alpha=A^n(x)$.

We call $\pv\in\PP(\RR^d)$ an \emph{invariant point} for $A$ if $A(i)\pv=\pv$ for every $i\in X$.
We call $\pv$ a \emph{$P$-expanding point} for $A$ if it is invariant for $A$ and there exist $l\ge 1$
and $c>0$ such that
\begin{equation}\label{eq.expanding_point}
\int_{[0;i]} \log \|DA^l(x,\pv)^{-1}\|^{-1} \, d\mu(x) \ge 4c p_i
\quad\text{for every $i\in X$.}
\end{equation}
Keep in mind that $\|\alpha^{-1}\|^{-1}$ is the \emph{co-norm} of an (invertible) matrix $\alpha$:
$$
\|\alpha^{-1}\|^{-1}
= \inf\big\{\frac{\|\alpha(v)\|}{\|v\|}: v \neq 0\big\}.
$$


\begin{remark}\label{r.atom_invariant}
If $\pv$ is an invariant point for $A$ and $\eta$ is a $P$-stationary vector then
$$
\eta_j(\{\pv\})
= \sum_{i\in X} \frac{p_{i}P_{i,j}}{p_j}\eta_{i}(\{\pv\})
\quad\text{for every $j\in X$.}
$$
Since $\sum_{i\in X} p_iP_{i,j}/p_j=1$ for every $j\in X$, this implies that the value of
$\eta_j(\{\pv\})$ is the same for all $j\in X$. In particular, if $\pv$ is an atom of $\eta$ then
it is an atom of $\eta_j$, with the same weight for every $j\in X$.
\end{remark}

\begin{proposition}\label{p.Lexpands}
Let $\pv\in\PP(\RR^d)$ be an invariant point for $A$ and suppose that there exist $a<b$ such that,
for $\mu$-almost every $x \in M$,
\begin{enumerate}
\item $\lim_n \frac 1n \log \|A^n(x)\vv\| \le a$ for any $\vv\in\pv$
\item $\lim_n \frac 1n \log \|A^n(x)\vw\| \ge b$ for any $\vw \notin \pv$.
\end{enumerate}
Then $\pv$ is a $P$-expanding point for $A$.
\end{proposition}

\begin{proof}
By assumption, $A(i)\pv=\pv$ for every $i\in X$. Let $A^\perp$ be the linear cocycle induced by $A$
on the orthogonal complement $\pv^\perp$, that is,
$$
A^\perp(x) \dot v = \proj_v A(x)\dot v \quad\text{for each $\dot v \in \pv^\perp$.}
$$
The second hypothesis implies (see $[17, \text{Proposition 4.14}]$) that the Lyapunov exponents of $A^\perp$ are
bounded below by $b$. In other words, the smallest Lyapunov exponent
$$
\lim_n \frac 1n \log \|(A^\perp)^n(x)^{-1}\|^{-1} \ge b
$$
for $\mu$-almost every $x \in M$. The expression \eqref{eq.proj_der} implies
$$
\|DA^n(x,\pv)^{-1}\|^{-1} \ge \frac{\|(A^\perp)^n(x)^{-1}\|^{-1}}{\|A^n(x)\vv\|}
$$
for any unit vector $\vv\in\pv$. Then, using the first hypothesis, it follows that
$$
\lim_n \frac 1n \log \|DA^n(x,\pv)^{-1}\|^{-1} \ge b - a > 0
$$
for $\mu$-almost every $x \in M$. It can be easily seen from \eqref{eq.derivative_estimate} that the sequence
on the left-hand side is uniformly bounded. So, we may use the bounded convergence theorem to conclude that
$$
\lim_n \frac 1n \int_{[0;i]} \log \|DA^n(x,\pv)^{-1}\|^{-1} \, d\mu(x) \ge (b - a)p_i
$$
for every $i\in X$. Fix any $c>0$ and then take $l\ge 1$ large enough so that
$$
\int_{[0;i]} \log \|DA^l(x,\pv)^{-1}\|^{-1} \, d\mu(x) \ge l\frac{b - a}{2}p_i\ge 4 c p_i
$$
for every $i\in X$.
\end{proof}

\begin{proposition}\label{p.expanding_delta}
Let $\pv\in\PP(\RR^d)$ be a $P$-expanding point for $A$ and $c>0$ and $l\ge 1$ be as in \eqref{eq.expanding_point}.
Then, assuming that $\delta>0$ is small enough,
$$
\int_{[0;i]} \|DA^l(x,\pv)^{-1}\|^\delta \, d\mu(x) \le (1 - 3c \delta)p_i
\quad\text{for every $i\in X$.}
$$
\end{proposition}

\begin{proof}
Fix $i\in X$. The function
$$
\delta\mapsto\psi_i(\delta) = \int_{[0;i]} \|DA^l(x,\pv)^{-1}\|^\delta \, d\mu(x)
$$
is differentiable and the derivative is given by
\begin{equation}\label{eq.der_formula}
\psi'_i(\delta) = \int_{[0;i]} \|DA^l(x,\pv)^{-1}\|^\delta \log \|DA^l(x,\pv)^{-1}\| \, d\mu(x).
\end{equation}
In particular, by the inequality \eqref{eq.expanding_point},
$$
\psi'_i(0) = \int_{[0;i]} \log \|DA^l(x,\pv)^{-1}\| \, d\mu(x)
\le -4cp_i.
$$
Now, the factor $\|DA^l(x,\pv)^{-1}\|^{\delta}$ in \eqref{eq.der_formula} is close to $1$,
uniformly in $x$ and $\dot\pv$, if $\delta$ is close to zero.
It follows that $\psi'_i(\delta) \le - 3cp_i$ for every small $\delta>0$.
Since $\psi_i(0)=p_i$, this gives the claim.
\end{proof}

\section{Proof of Theorem~\ref{t.main}}\label{s.proofmain}

The core of the proof of Theorem~\ref{t.main} is the following statement, whose proof will be given in
Sections~\ref{s.proof1} through~\ref{s.proof3}:

\begin{theorem}\label{t.key}
Suppose that $\pv \in \PP(\RR^d)$ is a $P$-expanding point for $A$.
Let $(A_k,P_k)_k$ be a sequence converging to $(A,P)$ and, for each $k\ge 1$, let $\eta_k$ be a $P_k$-stationary
unit vector. Suppose that $(\eta_k)_k$ converges to some $\eta$ in the weak$^*$ topology.
If $\pv$ is an atom for $\eta$ then $\eta_k$ must be atomic for all large $k$.
\end{theorem}

The second part of property \eqref{eq.Lexists1} in Proposition~\ref{p.Lexists} means that $L$ is an invariant point for $A$.
In view of property \eqref{eq.Lexists2} in that same proposition, Proposition~\ref{p.Lexpands}
implies that $L$ is also a $P$-expanding point for $A$.
By Theorem~\ref{t.key} applied to this point, $\eta_k$ has some atom for every large $k$.
Then, by Proposition~\ref{p.5}, there exists a finite set $\cL_k\subset\PP(\RR^2)$
such that $A_k(i)\pv\in\cL_k$ and $\eta_{k,i}(\{\pv\})>0$ for every $\pv\in\cL_k$ and every $i\in X$.

We claim that $\#\cL_k\le 2$ if $k$ is large. That can be seen as follows.
The existence of an invariant subspace $L$ implies that no matrix $A(i)$ is elliptic.
If all the $A(i)$ were either parabolic matrices (necessarily with the same eigenspace)
or multiples of the identity, then we would have $\lambda_-(A,P) = \lambda_+(A,P)$,
which is assumed not to be the case.
Thus, there exists $i\in X$ such that $A(i)$ is hyperbolic.
Since the set of hyperbolic matrices is open, it follows that $A_k(i)$ is hyperbolic for all large $k$.
Now it suffices to notice that for a hyperbolic $2 \times 2$ matrix,
any finite invariant set consists of either one or two eigenspaces.

Suppose that $\cL_k$ has a single element $L_k$. Define $\Phi_k:M\times\PP(\RR^2)\to\RR$ by
$\Phi_k(x,\pv)=\log (\|A_k(x)\vv\|/\|\vv\|)$. We claim that
\begin{equation}\label{eq.Lk}
\int \Phi_k(x,L_k) \, d\mu(x) = \lambda_+(A_k, P_k).
\end{equation}
To prove this, let $\zeta_k=(\zeta_{k,i})_{i\in X}$, where $\zeta_{k,i}$ is the Dirac mass at the
point $L_k$, for every $i\in X$. Observe that $\zeta_k$ is $P_k$-stationary, and so,
using the ergodic theorem for $\mu_k \ltimes \zeta_k$,
\begin{equation}\label{eq.fifi}
\int \Phi_k(x,L_k) \, d\mu_k(x) = \int \tilde\Phi_k(x,L_k) \, d\mu_k(x).
\end{equation}
We also have that $\tilde\Phi_k(x,\pv)=\lambda_+(A_k,P_k)$ for $(\mu_k\ltimes\eta_k)$-almost every $(x,\pv)$,
because
\begin{itemize}
\item $\tilde\Phi_k(x,\pv) \le \lambda_+(A_k,P_k)$ for $(\mu_k\ltimes\eta_k)$-almost every $(x,\pv)$;
\item our choice of $\eta_k$ means that $\int \tilde\Phi_k \, d(\mu_k\ltimes\eta_k) = \lambda_+(A_k,P_k)$.
\end{itemize}
Since $L_k$ is an atom for $\eta_k$, it follows that $\tilde\Phi_k(x,L_k) =\lambda_+(A_k,P_k)$
for $\mu_k$-almost every $x\in M$. Together with \eqref{eq.fifi}, this implies \eqref{eq.Lk}.

Up to restricting to a subsequence, we may assume that $(L_k)_k$ converges to some $L_0\in\PP(\RR^2)$.
Then $A(i)L_0=L_0$ for every $i\in X$. Let $\zeta_0=(\zeta_{0,i})_{i\in X}$ where $\zeta_{0,i}$ is the Dirac mass
at $L_0$ for every $i\in X$. Note that $\zeta_0$ is $P$-stationary, and so the ergodic theorem gives that
\begin{equation}\label{eq.fifizero}
\int \Phi(x,L_0) \, d\mu(x) = \int \tilde\Phi(x,L_0) \, d\mu(x).
\end{equation}
Clearly, $(\Phi_k)_k$ converges uniformly to $k$ when $k\to\infty$. So, \eqref{eq.Lk} yields
\begin{equation}\label{eq.Lzero}
\lambda_+(A_k,P_k) \to \int \Phi(x,L_0) \, d\mu(x).
\end{equation}
According to Proposition~\ref{p.Lexists}, there are two possibilities:
\begin{itemize}
\item[(i)] If $L_0 \neq L$ then the right-hand side of \eqref{eq.fifizero} is equal to $\lambda_+(A,P)$.
Then \eqref{eq.Lzero} means that $\lambda_+(A_k,P_k)$ converges to $\lambda_+(A,P)$,
which contradicts \eqref{eq.contradiction}.
\item[(ii)] If $L_0=L$ then the right-hand side of \eqref{eq.fifizero} is equal to $\lambda_-(A,P)$.
Then $\lambda_+(A_k,P_k) \to \lambda_-(A,P)$ and, by Lemma~\ref{l.somaOK},
$\lambda_-(A_k,P_k)\to\lambda_+(A,P)$.
This is a contradiction, as $\lambda_- \le \lambda_+$ and the inequality is strict at $(A,P)$.
\end{itemize}
We have shown that, for $k$ large, $\cL_k$ can not consist of a single point.

Now suppose that $\cL_k$ consists of two points, $L_k$ and $L_k'$. The arguments are similar.
Let $\zeta_k=(\zeta_{k,i})_{i\in X}$ with $\zeta_{k,i} = (\delta_{L_k} + \delta_{L_k'})/2$.
Arguing as we did for \eqref{eq.Lk}, we find that
\begin{equation}\label{eq.Lk2}
\int \frac 12 \big(\Phi_k(x,L_k)+\Phi_k(x,L_k')\big) \, d\mu_k(x)= \lambda_+(A_k,P_k).
\end{equation}
We may assume that $(L_k)_k$ and $(L'_k)_k$ converge to subspaces $L_0$ and $L_0'$, respectively.
Let $\zeta_0=(\zeta_{0,i})_{i\in X}$ with $\zeta_{0,i}=(\delta_{L_0}+\delta_{L_0'})/2$.
Then
\begin{equation}\label{eq.firifi}
\int \frac 12 \big(\Phi(x,L_0) + \Phi(x,L_0')\big) \, d\mu(x)
= \int \frac 12 \big(\tilde\Phi(x,L_0) + \tilde\Phi(x,L_0')\big) \, d\mu(x)
\end{equation}
and, taking the limit in \eqref{eq.Lk2},
$$
\lambda_+(A_k,P_k) \to \int \frac 12 \big(\Phi(x,L_0) + \Phi(x,L_0')\big) \, d\mu(x).
$$
If $L_0$ and $L_0'$ are both different from $L$ then the right-hand side of \eqref{eq.firifi}
is equal to $\lambda_+(A,P)$ and we reach a contradiction just as we did in case (i) of the previous paragraph.
Now suppose that $L_0=L$. We may assume that, for each $k$ large, there exists $i_k\in X$ such that
$A_k(i_k)L_k=L_k'$: otherwise, we could take $\cL_k=\{L_k\}$ instead, and that case has already
been dealt with. Passing to the limit along a convenient subsequence, we find $i_0\in X$ such that
$A(i_0)L_0=L_0'$. By the second part of \eqref{eq.Lexists1}, this implies that $L_0'$ is also equal to $L$.
Then the right-hand side of \eqref{eq.firifi} is equal to $\lambda_-(A,P)$, and we reach a contradiction
just as we did in case (ii) of the previous paragraph.
So, for $k$ large, $\cL_k$ can not consist of two points either.

This reduces the proof of Theorem~\ref{t.main} to proving Theorem~\ref{t.key}.

\section{Couplings and energy}\label{s.proof1}

In this section we prepare the proof of Theorem~\ref{t.key}.
Let $\pv$ be a $P$-expanding point for $A$ and assume that it is an atom for $\eta$.
According to Remark~\ref{r.atom_invariant}, there exists $\kappa>0$ such that $\eta_i(\{\pv\})=\kappa$
for every $i\in X$. Let $U\subset \PP(\RR^d)$ be an open neighborhood of $\pv$ such that
$$
\eta_i\big(\bar U\big) < \frac{10}{9} \kappa \quad\text{for every $i\in X$}.
$$

Fix constants $c>0$, $l\ge 1$ and $\delta>0$ as in Proposition~\ref{p.expanding_delta} and take $k$ to be
large enough that the conclusion of the proposition holds (further conditions will be imposed on $k$ along the way).
Since $A^l$ is continuous and $A_k^l$ converges to $A^l$ (uniformly) when $k\to\infty$, there exists an
open neighborhood $U_1 \subset U$ of $\pv$ such that
\begin{equation}\label{eq.expanding_delta1}
\sup_{\pz\in U_1}\|DA_k^l(x,\pz)^{-1}\|^\delta \le (1 + c \delta)\|DA^l(x,\pv)^{-1}\|^\delta
\end{equation}
for every $x \in M$ and every large $k$. Reducing $U_1$ if necessary, we may also assume that
\begin{equation}\label{eq.U1U}
A_k^l(x)^{-1}(\bar U_1) \subset U
\text{ for every $x\in M$ and every $k$ large.}
\end{equation}
Let $d$ be the distance on $\PP(\RR^d)$ defined by the angle between two directions, normalized in such
a way that the diameter is $1$. By the mean value theorem and \eqref{eq.expanding_delta1},
$$
\begin{aligned}
d(\pu,\pw)^\delta
& \le d(A_k^l(x)\pu,A_k^l(x)\pw)^\delta \, \sup_{z\in U_1} \|DA_k^l(x,\pz)^{-1}\|^\delta\\
& \le d(A_k^l(x)\pu,A_k^l(x)\pw)^\delta \, (1+c\delta) \, \|DA^l(x,\pv)^{-1}\|^\delta
\end{aligned}
$$
for any pair of distinct points $\pu, \pw \in U_1$, any $x\in M$ and any large $k$.
Then,
$$
\begin{aligned}
\int_{[0;i]} d(A_k^l(x)\pu, & A_k^l(x)\pw)^{-\delta} \, d\mu_k(x)\\
& \le (1+c\delta) \, d(\pu,\pw)^{-\delta} \int_{[0;i]} \|DA^l(x,\pv)^{-1}\|^\delta \, d\mu_k(x)\\
& \le (1+2c\delta) \, d(\pu,\pw)^{-\delta} \int_{[0;i]} \|DA^l(x,\pv)^{-1}\|^\delta \, d\mu(x)
\end{aligned}
$$
(the last inequality uses the fact that $\mu_k\to\mu$ and assumes that $k$ is large enough).
Using Proposition~\ref{p.expanding_delta}, we conclude that
\begin{equation}\label{eq.expanding_delta2}
\begin{aligned}
\int_{[0;i]} d(A_k^l(x)\pu,A_k^l(x)\pw)^{-\delta} \, d\mu_k(x)
& \le (1+2c \delta) \,  d(\pu,\pw)^{-\delta} \, (1-3c\delta) \, p_{k,i}\\
& \le (1-c\delta) \, p_{k,i} \, d(\pu,\pw)^{-\delta}
\end{aligned}
\end{equation}
for any $i\in X$, any pair of distinct points $\pu, \pw \in U_1$, and any large $k$.

For any measure $\xi$ on $\PP(\RR^d)^2$ and any $\delta>0$, define the \emph{$\delta$-energy}
of $\xi$ to be
$$
E_\delta(\xi)=\int \Psi \, d\xi,
$$
where
$$
\Psi(\pu,\pw)=\left\{\begin{array}{ll}d(\pu,\pw)^{-\delta} & \text{if } (\pu,\pw) \in U_1\times U_1\\
                                  1 & \text{otherwise.}\end{array}\right.
$$

Let $\pi_s:\PP(\RR^d)\times\PP(\RR^d) \to \PP(\RR^d)$ be the projection on the $s$-th coordinate, for $s=1, 2$.
The \emph{mass} $\|\eta\|$ of a measure $\eta$ on $\PP(\RR^d)$ is defined by $\|\eta\|=\eta(\PP(\RR^d))$.
If $\eta_1$ and $\eta_2$ are measures on $\PP(\RR^d)$ with the same mass, a \emph{coupling} of
$\eta_1$ and $\eta_2$ is a measure $\xi$ on $\PP(\RR^d)\times\PP(\RR^d)$ such that $(\pi_s)_*\xi=\eta_s$ for $s=1, 2$.
Define:
$$
e_\delta(\eta_1,\eta_2)=\inf\{E_\delta(\xi): \xi \text{ a coupling of $\eta_1$ and $\eta_2$}\}.
$$

A \emph{self-coupling} of a measure $\eta$ is a coupling of $\eta_1=\eta$ and $\eta_2=\eta$.
We call a self-coupling \emph{symmetric} if it is invariant under the involution
$\iota: (x,y) \mapsto (y,x)$. Define the \emph{$\delta$-energy} of $\eta$ to be
\begin{equation}\label{eq.edelta2}
e_\delta(\eta)
= e_\delta(\eta,\eta)
= \inf\{E_\delta(\xi): \xi \text{ a self-coupling of $\eta$}\}.
\end{equation}

\begin{remark}\label{r.semicont}
The  function $\xi \mapsto E_\delta(\xi)$ is lower semi-continuous. Indeed, it is not difficult
to find bounded continuous functions $\Psi_n:\PP(\RR^d)\times\PP(\RR^d)\to\RR$
increasing to $\Psi$ as $n\to\infty$. Then, given $(\xi_k)_k \to \xi$ and $\vep>0$, we may
fix $n$ such that $\int\Psi_n\,d\xi \ge \int E_\delta(\xi)-\vep$ (if the energy is infinite,
replace the right-hand side by $\vep^{-1}$). Then, by the definition of weak$^*$ topology,
$$
\liminf_k E_\delta(\xi_k)
\ge \liminf_k \int \Psi_n \, d\xi_k
\ge \int \Psi_n \, d\xi
\ge E_\delta(\xi) - \vep
$$
($\ge \vep^{-1}$ if the energy is infinite). Making $\vep\to 0$ one gets lower semi-continuity.

As a consequence, the infimum in \eqref{eq.edelta2} is always achieved.
Moreover, that remains true if we restrict to symmetric self-couplings: just note that if $\xi$ is a
self-coupling of $\eta$ then $\xi'=(\xi+\iota_*\xi)/2$ is a symmetric self-coupling of $\eta$ and
$E_\delta(\xi)= E_\delta(\xi')$. As a further consequence, the function $\eta \mapsto e_\delta(\eta)$
is lower semi-continuous. Indeed, given $(\eta_k)_k\to\eta$, take self-couplings $(\xi_k)_k$ with
$E_\delta(\xi_k)=e_\delta(\eta_k)$. Up to restricting to a subsequence, $(\xi_k)_k$ converges to
some $\xi$, which is necessarily a self-coupling of $\eta$, and
$$
e_\delta(\eta) \le E_\delta(\xi) \le \liminf_k E_\delta(\xi_k)=\liminf_k e_\delta(\eta_k).
$$
This proves lower semi-continuity.
\end{remark}

\begin{lemma}\label{l.fatatoms}
For any measure $\eta$ in $\PP(\RR^d)$ and any $\delta>0$:
\begin{enumerate}
\item If $\eta(\{\pu\})>\|\eta\|/2$ for some $\pu \in U_1$ then $e_\delta(\eta)=\infty$.
\item If $\eta(\{\pu\})<\|\eta\|/2$ for every $\pu \in \bar U_1$ then $e_\delta(\eta)<\infty$.
\end{enumerate}
\end{lemma}

\begin{proof}
First we prove (1). Take $\pu$ as in the hypothesis. For any self-coupling $\xi$ of $\eta$, we have
$$
\xi(\{\pu\}^c \times \PP(\RR^d)) = \eta(\{\pu\}^c) = \xi(\PP(\RR^d) \times \{\pu\}^c).
$$
Taking the union, $\xi(\{(\pu,\pu)\}^c) \le 2\eta(\{\pu\}^c) < \eta(\PP(\RR^d))=\xi(\PP(\RR^d)^2)$.
This implies that $(\pu,\pu)$ is an atom for $\xi$. Hence $E_\delta(\xi)=\infty$,
as claimed in (1).

A result analogous to (2) was proven in $[17, \text{Lemma 10.7}]$, for a slightly different notion of energy.
Exactly the same arguments can be used in the present situation to construct a self-coupling $\xi$ of
$\eta$ which gives zero weight to a neighborhood of $\{(\pu,\pu):\pu\in\bar U_1\}$ and, hence,
has finite $\delta$-energy.
\end{proof}

Let $\eta=(\eta_i)_{i\in X}$ be a measure vector on $\PP(\RR^d)$ and $\xi=(\xi_i)_{i\in X}$
be a measure vector on $\PP(\RR^d)\times\PP(\RR^d)$. We say that $\xi$ is a (symmetric) self-coupling of $\eta$
if $\xi_i$ is a (symmetric) self-coupling of $\eta_i$ for every $i\in X$. Then we define
$$
E_\delta(\xi) = \sum_{i\in X} p_i E_\delta(\xi_i)
$$
and
$$
e_\delta(\eta)
 = \inf\big\{E_\delta(\xi): \xi \text{ a self-coupling of } \eta\big\}
 = \sum_{i\in X} p_i e_\delta(\eta_i)
$$
(of course, this depends on the probability vector $p$). It follows from Remark~\ref{r.semicont} that the
function $\eta \mapsto e_\delta(\eta)$ is lower semi-continuous in the space of $P$-stationary measure vectors.
Moreover, for $(\eta_k)_k$ as in Theorem~\ref{t.key},
\begin{equation}\label{eq.semicont}
\liminf_k e_\delta(\eta_k)
= \liminf_k \sum_{i\in X} p_{k,i} e_\delta(\eta_{k,i})
\ge \sum_{i\in X} p_i e_\delta(\eta_i)
= e_\delta(\eta).
\end{equation}

\section{The large atom case}\label{s.proof2}

We split the proof of Theorem~\ref{t.key} into two cases, depending on the value of $\kappa=\eta_i(\{\pv\})$.
In this section we treat the case $\kappa>9/10$.

\begin{proposition}\label{p.couplings_decay1}
There exists $C>0$ such that, for every large $k$, given any symmetric self-coupling $\xi_k$ of $\eta_k$,
there exists a symmetric self-coupling $\xi_k''$ of $\eta_k$ satisfying
$$
E_\delta(\xi_k'') \leq C+(1-c\delta) E_\delta(\xi_k).
$$
\end{proposition}

The present case of Theorem~\ref{t.key} can be deduced as follows.
Suppose that there exists a subsequence $(k_j)_j\to\infty$ such that $\eta_{k_j}$ has no atoms and, hence,
$e_\delta(\eta_{k_j})$ is finite. Then, Proposition~\ref{p.couplings_decay1} gives that
$$
e_\delta(\eta_{k_j}) \leq C/(c\delta)
\quad\text{for every large $j$.}
$$
By the lower semi-continuity property \eqref{eq.semicont}, it follows that $e_\delta(\eta)\le C /(c\delta)$.
However, by Lemma~\ref{l.fatatoms}, the assumption $\kappa > 9/10$ implies that $e_\delta(\eta)=\infty$.
This contradiction proves that $\eta_k$ does have some atom for every large $k$.

In the remainder of this section we prove Proposition~\ref{p.couplings_decay1}.
Let $\xi_k$ be a self-coupling of $\eta_k$. The construction of the self-coupling $\xi_k''$
will be done in two steps.
Fix open neighborhoods $U_3 \subset \bar U_3 \subset U_2 \subset \bar U_2 \subset U_1$ of the point $\pv$.
The first step is:

\begin{lemma}\label{l.step11}
There exists $C_1>0$ such that, assuming $k$ is sufficiently large, given any self-coupling $\xi_k$ of $\eta_k$
there exists a self-coupling $\xi_k'$ of $\eta_k$ satisfying
\begin{enumerate}
\item[(a)] $\xi_{k,i}'(U_2 ^c \times U_2^c)=0$  for every $i\in X$;
\item[(b)] $E_\delta(\xi_{k,i}') \le E_\delta(\xi_{k,i})+C_1$ for every $i\in X$.
\end{enumerate}
\end{lemma}

\begin{proof}
For each $i\in X$, let $\zeta_{k,i}$ be the projection of $(\xi_{k,i} \mid U_2^c \times U_2^c)$
and $\eta_{k,i}^3$ be the projection of $(\xi_{k,i} \mid U_3 \times U_3)$, on either coordinate.
Then take
$$
\begin{aligned}
\xi_{k,i}'
= \xi_{k,i}
- (\xi_{k,i} \mid U_2^c \times U_2^c)
& - \frac{\|\zeta_{k,i}\|}{\|\eta_{k,i}^3\|} (\xi_{k,i} \mid U_3 \times U_3)\\
& + \frac{1}{\|\eta_{k,i}^3\|} \big[\zeta_{k,i} \times \eta_{k,i}^3 + \eta_{k,i}^3 \times \zeta_{k,i}\big]
\end{aligned}
$$
Assuming $k$ is sufficiently large, $\eta_{k,i}(U_3) > (8/9) \kappa > 8/10$ for all $i\in X$.
Then, $\|\xi_{k,i} \mid U_3 \times U_3\| > 6/10$. It follows that
$$
\|\eta_{k,i}^3\|
= \|\xi_{k,i} \mid U_3 \times U_3\|
> \frac{6}{10}
\quand
\|\zeta_{k,i}\|
\le \|\eta_{k,i} \mid U_2^c\|
< \frac{2}{10}.
$$
In particular, $\|\zeta_{k,i}\| < \|\eta_{k,i}^3\|$, which ensures that $\xi_{k,i}$ is a positive measure.
Next, notice that the projection of $\xi_{k,i}'$ on either coordinate is equal to
$$
\eta_{k,i} - \zeta_{k,i} - \frac{\|\zeta_{k,i}\|}{\|\eta_{k,i}^3\|} \eta_{k,i}^3
+ \zeta_{k,i} + \frac{\|\zeta_{k,i}\|}{\|\eta_{k,i}^3\|} \eta_{k,i}^3
= \eta_{k,i}.
$$
So, $\xi_{k,i}'$ is indeed a self-coupling of $\eta_{k,i}$. Finally, it is clear that
\begin{equation}\label{eq.xilinha}
E_\delta(\xi_{k,i}') \le E_\delta(\xi_{k,i}) +
\frac{1}{\|\eta_{k,i}^3\|} E_\delta(\zeta_{k,i} \times \eta_{k,i}^3 + \eta_{k,i}^3 \times \zeta_{k,i}).
\end{equation}
The measure $(\zeta_{k,i} \times \eta_{k,i}^3 + \eta_{k,i}^3 \times \zeta_{k,i})$ is concentrated on the set
$(U_2^c \times U_3 \cup U_3 \times U_2^c)$, which is at positive distance from the diagonal.
Thus, its energy is uniformly bounded. We have already seen that $\|\eta_{k,i}^3\|$ is bounded away from zero.
Thus, the last term in \eqref{eq.xilinha} is uniformly bounded, and that yields the claim in the lemma.
\end{proof}

Let $\cP_k$ denote the operator defined as in \eqref{eq.operator1}, with $(A,P)$ replaced with $(A_k,P_k)$.
The second, and final step of our construction uses the \emph{diagonal action} $\fP_k$ of $\cP_k$,
defined as follows:
$$
(\fP_k\xi)_j(D_1\times D_2) = \sum_{i\in X} \frac{p_{k,i} P_{k,i,j}}{p_{k,j}}\xi_i\big(A_k(i)^{-1}(D_1)\times A_k(i)^{-1}(D_2)\big)
$$
for any $j\in X$ and any measurable $D\subset\PP(\RR^d)$. Equivalently,
\begin{equation}\label{eq.operator2bis}
p_{k,j} \int \Phi \, d(\fP_k\xi)_j
= \sum_{i \in X} \int_{[0;i,j]} \int \Phi\big(A_k(x)\pu, A_k(x)\pw\big) \, d\xi_i(\pu,\pw) \, d\mu_k(x)
\end{equation}
for any measurable function $\Phi:\PP(\RR^d)\times\PP(\RR^d)\to[0,\infty)$ (the two sides
of the equality may be infinite). It is clear that
\begin{equation}\label{eq.operator_tris}
(\pi_s)_* \circ \fP_k = \cP_k \circ (\pi_s)_*
\quad\text{for $s=1, 2$.}
\end{equation}
Take $\xi_k'' = \fP_k^l \xi_k'$. Notice that $\xi''_k$ is a symmetric measure and a self-coupling of $\eta_k$:
the previous relation gives that
$$
(\pi_s)_* \xi''_k
= \cP_k^l ((\pi_s)_* \xi'_k)
= \cP_k^l \eta_k = \eta_k
\quad\text{for $s=1, 2$.}
$$

\begin{lemma}\label{l.step12}
There exists $C_2>0$ such that, assuming $k$ is sufficiently large,
$$
E_\delta(\xi_k'') \le (1-c\delta) E_\delta(\xi_k') + C_2.
$$
\end{lemma}

\begin{proof}
From the definition of $\fP_k$ in \eqref{eq.operator2bis} we get that
$$
p_{k,j} \int \Psi \, d\xi_{k,j}''\\
 = \sum_{i\in X} \int_{[0;i]\cap [l;j]} \int \Psi\big(A_k^l(x)\pu,A_k^l(x)\pw\big) \, d\xi_{k,i}' \big(\pu,\pw\big) \, d\mu_k\big(x\big).
$$
for every $j\in X$. Thus, adding over $j$,
\begin{equation}\label{eq.Psi1}
E_\delta(\xi''_k)
 = \sum_{i\in X} \int_{[0;i]} \int \Psi\big(A_k^l(x)\pu,A_k^l(x)\pw\big) \, d\xi_{k,i}' \big(\pu,\pw\big) \, d\mu_k\big(x\big).
\end{equation}
If $(\pu,\pw)\in (U_1\times U_1)$,  then $\Psi(\pu,\pw)=d(\pu,\pw)^{-\delta}$.
Using \eqref{eq.expanding_delta2}, we get that
\begin{equation}\label{eq.Psi2}
\begin{aligned}
\int_{[0;i]} \Psi(A_k^l(x)\pu,A_k^l(x)\pw)\, d\mu_k(x)
& \le \int_{[0;i]} d(A_k^l(x)\pu,A_k^l(x)\pw)^{-\delta} \, d\mu_k(x)\\
& \le (1-c \delta) \, p_{k,i} \, \Psi(\pu,\pw).
\end{aligned}
\end{equation}
Now suppose that $(\pu,\pw)\notin (U_1\times U_1)$. By the property in Lemma~\ref{l.step11}(a),
we only need to consider $(\pu,\pw)\notin (U_2^c\times U_2^c)$.
Then $(\pu,\pw) \in U_2 \times U_1^c \cup U_1^c \times U_2$. Since the latter set is at positive
distance from the diagonal, it follows that $d(\pu,\pw)$ is uniformly bounded from below.
Hence, there exists $C_2>0$ such that
\begin{equation}\label{eq.Psi3}
\Psi(A_k^l(x)\pu,A_k^l(x)\pw) \le d(A_k^l(x)\pu,A_k^l(x)\pw)^{-\delta}
\le C_2.
\end{equation}
for every $x\in M$. Using \eqref{eq.Psi2} and \eqref{eq.Psi3}, together with Fubini, we get that
\eqref{eq.Psi1} yields
$$
E_\delta(\xi''_k)
 \le \sum_{i\in X} \int_{I_1} (1-c\delta) p_{k,i} \Psi(\pu,\pw) \, d\xi'_{k,i}(\pu,\pw)
                 + \int_{I_1^c} C_2 p_{k,i} \, d\xi'_{k,i}(\pu,\pw)
$$
where $I_1=\{(\pu,\pw)\in U_1 \times U_1\}$. Then,
$$
E_\delta(\xi''_k)
 \le (1-c\delta) \sum_{i\in X} p_{k,i} E_\delta(\xi'_{k,i}) + C_2 p_{k,i}
 = (1-c\delta) E_\delta(\xi'_k) + C_2
$$
as claimed.
\end{proof}

Combining Lemmas~\ref{l.step11} and~\ref{l.step12}, we find that
$$
E_\delta(\xi_k'') \le (1-c\delta) E_\delta(\xi_k) + C
\quad\text{with } C = (1-c\delta) C_1 + C_2.
$$
This finishes the proof of Proposition~\ref{p.couplings_decay1}.

\section{The small atom case}\label{s.proof3}

We are left to consider $\kappa \le 9/10$. This is similar to the previous case, but the analysis
must be localized on suitable neighborhoods of $\pv$, inside which the relative weight of the atom
is close to $1$, and the precise choice of such neighborhoods turns out to be much more delicate
than one would anticipate. We do as follows.

Since $\kappa \le 9/10$, we have $\eta_{k,i}(\bar U) < (10/9)\kappa \le 1$ for every $i\in X$ and every large $k$.
Suppose that $\eta_k$ is non-atomic for arbitrarily large values of $k$. Up to restricting to a subsequence,
we may suppose that this is the case for every $k$. Then every $\eta_{k,i}$ is a continuous measure and so
we may find $V_{k,i}\supset\bar U$ such that $\eta_{k,i}(V_{k,i}) =  (10/9)\kappa$. We denote
$$
V_k=(V_{k,i})_{i\in X}
\quand
\eta_k \mid V_k = (\eta_{k,i} \mid V_{k,i})_{i\in X}
\quad\text{for each $k\ge 1$.}
$$
The construction we have just described is designed to get the following fact, which will be needed in a while:

\begin{lemma}\label{l.inout}
$\cP^l_k(\eta_k \mid V_k^c)_j(V_{k,j}) = \cP^l_k(\eta_k \mid V_k)_j(V_{k,j}^c)$ for every $j\in X$.
\end{lemma}

\begin{proof}
Clearly,
\begin{equation}\label{eq.inout1}
\begin{aligned}
\cP_k^l(\eta_k\mid V_k)_j(\PP(\RR^2))
& = \cP^l_k(\eta_k \mid V_k)_j(V_{k,j}) + \cP^l_k(\eta_k \mid V_k)_j(V_{k,j}^c)\\
\eta_{k,j}(V_{k,j})
  = (\cP^l_k \eta_k)_j(V_{k,j})
& = \cP^l_k(\eta_k \mid V_k)_j(V_{k,j}) + \cP^l_k(\eta_k \mid V_k^c)_j(V_{k,j}).
\end{aligned}
\end{equation}
Now, the definition of $\cP_k$ in \eqref{eq.operator1} gives that
$$
\begin{aligned}
\cP_k^l(\eta_k\mid V_k)_j(\PP(\RR^2))
& = \sum_{i\in X} \frac{1}{p_{k,j}} \int_{[0;i]\cap[l;j]} \eta_{k,i}\big(A_k^l(x)^{-1}(\PP(\RR^2)) \cap V_{k,i}\big) \, d\mu_k(x)\\
& = \sum_{i\in X} \frac{1}{p_{k,j}} \int_{[0;i]\cap[l;j]} \eta_{k,i}\big(V_{k,i}\big) \, d\mu_k(x).
\end{aligned}
$$
The key observation is that, by construction, $\eta_{k,i}(V_{k,i})$ is independent of $i$ and $k$.
Thus the previous equality may be rewritten as
\begin{equation}\label{eq.inout2}
\begin{aligned}
\cP_k^l(\eta_k\mid V_k)_j(\PP(\RR^2))
& = \sum_{i\in X} \frac{1}{p_{k,j}} \int_{[0;i]\cap[l;j]} \eta_{k,j}\big(V_{k,j}\big) \, d\mu_k(x)\\
& = \frac{1}{p_{k,j}} \int_{[l;j]} \eta_{k,j}\big(V_{k,j}\big) \, d\mu_k(x)
= \eta_{k,j}(V_{k,j}).
\end{aligned}
\end{equation}
The claim is a direct consequence of \eqref{eq.inout1} and \eqref{eq.inout2}.
\end{proof}

We are going to prove the following localized version of Proposition~\ref{p.couplings_decay1}:

\begin{proposition}\label{p.couplings_decay2}
There exists $C>0$ such that, for every large $k$, given any symmetric self-coupling $\xi_k$ of $\eta_k \mid V_k$,
there exists a symmetric self-coupling $\xi''_k$ of $\eta_k \mid V_k$ satisfying
$$
E_\delta(\xi''_k) \leq C+(1-c\delta) E_\delta(\xi_k).
$$
\end{proposition}

The present case of Theorem~\ref{t.key} may be deduced as follows. Since the $\eta_k$ have no atoms,
the energy $e_\delta(\eta_k \mid V_k)$ is finite. Then, Proposition~\ref{p.couplings_decay2} gives that
$$
e_\delta(\eta_k \mid V_k) \leq C/(c\delta)
\quad\text{for every large $j$.}
$$
Let $\hat\eta=(\hat\eta_i)_{i\in X}$ be any accumulation point of $(\eta_k \mid V_k)$.
Then $e_\delta(\hat\eta)\le C /(c\delta)$, by the lower semi-continuity property \eqref{eq.semicont}.
For any small open neighborhood $W$ of $L$ and any $i\in X$,
$$
\hat\eta_i(\bar W)
\ge \limsup_k \eta_{k,i}(V_{k,i} \cap \bar W)
= \limsup_k \eta_{k,i}(\bar W)
\ge \liminf_k \eta_{k,i}(W)
\ge \eta_i(W)
$$
because $\eta_{k,i}\mid V_{k,i}$ converges to $\hat\eta_i$ and $\eta_{k,i}$ converges to $\eta_i$ in the weak$^*$ topology.
Making $W \to \{L\}$, we obtain that $\hat\eta_i(\{L\})\ge \kappa = ({9}/{10}) \|\hat\eta_i\|$ for all $i\in X$.
This implies that $e_\delta(\hat\eta)=\infty$, contradicting the previous conclusion.
This contradiction proves that $\eta_k$ does have some atom for all $k$ sufficiently large.

In the remainder of this section we prove Proposition~\ref{p.couplings_decay2}.
Let $U, U_1, U_2, U_3$ be as before. The first step is:

\begin{lemma}\label{l.step21}
There exists $C_1>0$ such that, assuming $k$ is sufficiently large, given any self-coupling $\xi_k$ of $\eta_k\mid V_k$
there exists a self-coupling $\xi_k'$ of $\eta_k\mid V_k$ satisfying
\begin{enumerate}
\item[(a)] $\xi_{k,i}'\big((V_{k,i}\setminus U_2) \times (V_{k,i}\setminus U_2))=0$  for every $i\in X$;
\item[(b)] $E_\delta(\xi_{k,i}') \le E_\delta(\xi_{k,i})+C_1$ for every $i\in X$.
\end{enumerate}
\end{lemma}

\begin{proof}
Analogous to Lemma~\ref{l.step11}, with $V_{k,i}\setminus U_2$ in the place of $U_2^c$.
\end{proof}

We also have the following variation of Lemma~\ref{l.step12}:

\begin{lemma}\label{l.step22}
There exists $C_2>0$ such that, assuming $k$ is sufficiently large,
$$
E_\delta(\fP_k^l\xi_k') \le (1-c\delta) E_\delta(\xi_k') + C_2.
$$
\end{lemma}

\begin{proof}
Analogous to Lemma~\ref{l.step12}, with $V_{k,i}\setminus U_2$ in the place of $U_2^c$.
\end{proof}

Now we come to a main difference with respect to the more global situation treated in Section~\ref{s.proof2}.
Using \eqref{eq.operator_tris} one gets that
$$
\tilde\xi_k = \fP_k^l\xi_k' \text{  is a symmetric self-coupling of }
\cP^l_k(\eta_k \mid V_k).
$$
However, this \emph{does not} mean that $\tilde\xi_k$ is a self-coupling of $(\eta_k \mid V_k)$
as the latter measure need not be $P_k$-stationary.

To bypass this difficulty, we begin by relating $(\eta_k \mid V_k)$ to the projection of $\tilde\xi_k$
or, more precisely, to
$$
\eta_k^1 = \text{ projection of } (\tilde\xi_k \mid V_k \times V_k).
$$
The next lemma shows that $\eta_k^1 \le (\eta_k \mid V_k)$ and describes the difference
between the two vectors.

\begin{lemma}\label{l.couplings}
$(\eta_k \mid V_k) = \eta_k^1 + I_k + O_k$ with
\begin{enumerate}
\item $I_k =$ restriction of $\cP_k^{l}(\eta_k \mid V_k^c)$ to $V_k$;
\item $O_k = $ is the projection of $(\tilde\xi_k \mid V_k \times V_k^c)$ on the first coordinate;
\item $\|O_{k,j}\| \le \|I_{k,j}\| \le (2/9)\kappa$ for every $j\in X$ and $k$ sufficiently large.
\end{enumerate}
\end{lemma}

\begin{proof}
On the one hand,
$$
\pi_{1*}\tilde\xi_k
= \cP_k^l(\eta_k \mid V_k)
= \cP_k^l(\eta_k) - \cP_k^l (\eta_k \mid V_k^c)
= \eta_k - \cP_k^l(\eta_k \mid V_k^c)
$$
because $\eta_k$ is $\cP_k^l$-stationary. Restricting to $V_k$ we get that
$$
(\pi_{1*}\tilde\xi_k) \mid V_k = (\eta_k \mid V_k) - I_k.
$$
On the other hand,
$$
(\pi_{1*}\tilde\xi_k) \mid V_k
 = \pi_{1*}\big(\tilde\xi_k \mid V_k \times V_k\big)
+ \pi_{1*}\big(\tilde\xi_k \mid V_k \times V_k^c\big)
 = \eta_k^1 + O_k.
$$
Combining two equalities one gets that $(\eta_k \mid V_k) = \eta_k^1 + I_k + O_k$.

We are left to proving the estimates in part (3) of the statement.
Let $j\in X$ be fixed. By Lemma~\ref{l.inout},
$$
\|I_{k,j}\|
= \cP^l_k(\eta_k \mid V_k^c)_j(V_{k,j})
= \cP^l_k(\eta_k \mid V_k)_j(V_{k,j}^c).
$$
Since $\tilde\xi_k$ is a self-coupling of $\cP^l_k(\eta_k \mid V_k)$,
$$
\cP^l_k(\eta_k \mid V_k)_j(V_{k,j}^c)
= \tilde\xi_k(\PP(\RR^2) \times V_{k,j}^c)
\ge \tilde\xi_k(V_{k,j} \times V_{k,j}^c)
= \|O_{k,j}\|.
$$
These two inequalities ensure that $\|O_{k,j}\| \le \|I_{k,j}\|$.

Finally, the first part of the lemma implies that $I_{k,j} \leq \eta_{k,j} \mid V_{k,j}$
for every $j\in X$. Since $\eta_{k,j}(V_{k,j})=(10/9)\kappa$ and
$\liminf_k \eta_{k,j}(U_1) \ge \eta_j(U_1) \ge \kappa$, this implies that
$$
I_{k,j}(V_{k,j}\setminus U_1)
\le \eta_{k,j}(V_{k,j}\setminus U_1)
< \frac{2}{9}\kappa
\quad\text{for every large $k$.}
$$
The condition \eqref{eq.U1U} ensures that $A_k^l(x)^{-1}(U_1)\subset U \subset V_{k,i}$
for every $x \in [0;i]$, every $i \in X$ and every large $k$. Thus,
$$
\begin{aligned}
I_{k,j}(U_1)
& = \cP^l_k(\eta_k \mid V_k^c)_j(U_1)\\
& = \sum_{i\in X} \frac{1}{p_{k,j}} \int_{[0;i]\cap [l;j]} \eta_{k,i}\big(V_{k,i}^c \cap A_k^l(x)^{-1}(U_1)\big) \, d\mu_k(x)
= 0
\end{aligned}
$$
for every large $k$. These two facts imply that
$\|I_{k,j}\| = I_{k,j}(V_{k,j} \setminus U_1) \le (2/9)\kappa$, as claimed.
\end{proof}

According to Lemma~\ref{l.couplings}, the difference $(\eta_k \mid V_k) - \eta_k^1$ is positive
and relatively small. This suggests that we try and construct the self-coupling $\xi''_k$ of
$(\eta_k \mid V_k)$ we are looking for by adding suitable correcting terms to the restriction
of $\tilde\xi_k$ to the $V_k \times V_k$.
These correcting terms should be concentrated outside a neighborhood of the diagonal,
if possible, so that their contribution to the total energy is bounded.
We choose $\xi''_k=(\xi''_{k,i})_{i\in X}$, with
\begin{equation}\label{eq.barxin}
\begin{aligned}
\xi''_{k,i}
= & \big[(\tilde \xi_{k,i} \mid V_{k,i} \times V_{k,i})-\frac {\|\zeta_{k,i}\|} {\|\eta_{k,i}^3\|} (\tilde\xi_{k,i} \mid U_3 \times U_3)\big]\\
  & \qquad + \frac {1}{\|I_{k,i}\|} \big[(O_{k,i} \mid U_2) \times I_{k,i} + I_{k,i} \times (O_{k,i} \mid U_2)\big]\\
  & \qquad + \frac {1} {\|\eta_{k,i}^3\|} \big[(\zeta_{k,i} \times \eta_{k,i}^3)+(\eta_{k,i}^3\times \zeta_{k,i})\big]
\end{aligned}
\end{equation}
where $\eta_{k,i}^3$ is the projection of $(\tilde \xi_{k,i} \mid U_3 \times U_3)$ on either coordinate and
$$
\zeta_{k,i}= \big(1-\frac {\|O_{k,i} \mid U_2\|} {\|I_{k,i}\|}\big) I_{k,i} + (O_{k,i} \mid U_2^c).
$$

\begin{lemma}\label{l.iscoupling}
$\xi''_k$ is a symmetric self-coupling of $(\eta_k \mid V_k)$.
\end{lemma}

\begin{proof}
Let $i\in X$ be fixed.
It is clear that the three terms on the right-hand side of \eqref{eq.barxin} are symmetric.
Moreover,
$$
\|O_{k,i} \mid U_2\| \le \|O_{k,i}\| \le \|I_{k,i}\|.
$$
This ensures that $\zeta_{k,i}$ is a positive measure. Clearly, the last  two terms in \eqref{eq.barxin} are positive.
To see that  the first term is also positive, it suffices to check that
$\|\zeta_{k,i}\| < \|\eta_{k,i}^3\|$. That can be done as follows.

Fix an open neighborhood $U_4 \subset U_3$ such that $A^l(x)\bar U_4 \subset U_3$ for every $x\in M$.
Assuming $k$ is large enough,
$$
\eta_{k,i}(U_4)
> \frac{8}{9}\eta_{i}(U_4)
\ge \frac{8}{9}\kappa
\quad\text{for every $i\in X$.}
$$
Since $\eta_{k,i}(V_{k,i})=(10/9)\kappa$, this implies that $\xi'_{k,i}(U_4 \times U_4)\ge (5/9)\kappa$.
Increasing $k$ again, if necessary, $A_k^l(x)^{-1}(U_3) \supset U_4$ for every $x\in M$. Hence,
$$
\|\eta^3_{k,j}\|
= \tilde\xi_{k,j}(U_3 \times U_3)
\ge \sum_{i\in X} \frac{1}{p_{k,j}} \int_{[0;i]\cap[l;j]} \xi'_{k,i}(U_4 \times U_4) \, d\mu_k(x)
\ge \frac{5}{9}\kappa.
$$
So, the claim that $\|\zeta_{k,i}\| \leq \|\eta^3_{k,i}\|$ is an immediate consequence of
$$
\|\zeta_{k,i}\|
\leq \|I_{k,i}\| + \|O_{k,i}\|
\leq \frac{4}{9}\kappa.
$$

Next, observe that the projection of $\xi_{k,i}''$ on either coordinate is equal to
$$
\big[\eta_{k,i}^1-\frac {\|\zeta_{k,i}\|} {\|\eta_{k,i}^3\|} \eta_{k,i}^3\big]
+ \big[(O_{k,i} \mid U_2)+\frac {\|O_{k,i} \mid U_2\|} {\|I_{k,i}\|} I_{k,i}\big]
+ \big[\zeta_{k,i} + \frac {\|\zeta_{k,i}\|} {\|\eta_{k,i}^3\|} \eta_{k,i}^3\big]
$$
and that adds up to $\eta_{k,i}^1 + I_{k,i} + O_{k,i} = \eta_{k,i} \mid V_{k,i}$.
So, $\xi''_{k,i}$ is a self-coupling of $\eta_{k,i} \mid V_{k,i}$ as we wanted to prove.
\end{proof}

\begin{lemma}\label{l.step23}
There exists $C_3>0$ such that
$$
E_\delta(\xi''_{k,i}) \le E_\delta(\tilde\xi_{k,i}) + C_3
\quad\text{for every $i\in X$ and $k$ sufficiently large.}
$$
\end{lemma}

\begin{proof}
Notice that the supports of the two correcting terms
$$
\frac {1}{\|I_{k,i}\|} \big[(O_{k,i} \mid U_2) \times I_{k,i} + I_{k,i} \times (O_{k,i} \mid U_2)\big]
$$
and
$$
\frac {1} {\|\eta_{k,i}^3\|} \big[(\zeta_{k,i} \times \eta_{k,i}^3)+(\eta_{k,i}^3\times \zeta_{k,i})\big]
$$
are uniformly away from the diagonal of $\PP(\RR^d)^2$.
Indeed, $O_{k,i}\mid U_2$ is concentrated on $U_2$ while $I_{k,i}$ is concentrated on $U_1^c$, as we have seen.
Similarly, $\zeta_{k,i}$ is concentrated on $U_2^c$ while $\eta_{k,i}^3$ is concentrated on $U_3$.
Moreover, the masses of these two correcting terms are uniformly bounded (by $1$, say).
Thus, the $\delta$-energy of their sum is bounded by some constant $C_3>0$, for every large $k$.
Finally, it is clear that the $\delta$-energy of the first term in \eqref{eq.barxin} is bounded by
$E_\delta(\tilde\xi_{k,i} \mid V_{k,i} \times V_{k,i})$.
Therefore,
$$
E_\delta(\xi''_{k,i})
\leq E_\delta(\tilde\xi_{k,i} \mid V_{k,i} \times V_{k,i}) + C_3
\leq E_\delta(\tilde\xi_{k,i})+C_3,
$$
as claimed.
\end{proof}

By Lemmas~\ref{l.step21}, \ref{l.step22} and~\ref{l.step23},
$$
E_\delta(\xi''_k)
\le (1-c\delta) E_\delta(\xi_k) + (1-c\delta) C_1 + C_2 + C_3.
$$
for every large $k$. This proves Proposition~\ref{p.couplings_decay2} for $C=(1-c\delta)C_1 + C_2 + C_3$.

The proof of Theorem~\ref{t.key} is complete.

\renewcommand{\bibname}{References}
\addcontentsline{toc}{chapter}{\bibname}       

\end{document}